\documentclass{amsart}
\usepackage[T1]{fontenc}
\usepackage[latin1]{inputenc}
\usepackage{amsmath}
\usepackage{amsfonts}
\usepackage{amssymb}        
\usepackage[numbers]{natbib} 
\usepackage[all, cmtip]{xy}
\usepackage{bm}

\def\subclassname{{\bfseries Mathematics Subject Classification
(2010)}\enspace}
\def\subclass#1{\par\addvspace\medskipamount{\rightskip=0pt plus1cm
\def\and{\ifhmode\unskip\nobreak\fi\ $\cdot$
}\noindent\subclassname\ignorespaces#1\par}}

   \newcommand{\dx}{\mathrm{d}}
   \DeclareMathOperator{\N}{\mathbb{N}}
   \DeclareMathOperator{\R}{\mathbb{R}}
   \DeclareMathOperator{\C}{\mathbb{C}}
   
   \DeclareMathOperator{\supp}{supp}
   \DeclareMathOperator{\om}{\mathcal{O}_M}
   \DeclareMathOperator{\oc}{\mathcal{O}_C}
   \DeclareMathOperator{\bp}{\dot{\mathcal{B}}}

\newtheorem{proposition}{Proposition}
\newtheorem{lemma}{Lemma}

\theoremstyle{remark}
\newtheorem{remark}{Remark}
\numberwithin{theorem}{section}

\begin{document}
\title[Multiplications and Convolutions in distribution spaces]{Multiplications and Convolutions in {L}. {S}chwartz' Spaces of Test Functions and Distributions and their Continuity}
\author{Julian Larcher}
\address{Department of Mathematics, University of Innsbruck, Technikerstr. 13, A-6020 Innsbruck, Austria}
\email{julian.larcher@uibk.ac.at}

\begin{abstract}
We list multiplier and convolutor spaces of the spaces occurring in L. Schwartz' ``Th\'{e}orie des distributions''. Furthermore we clarify whether the multiplications and convolutions are continuous or not.
\end{abstract} 

\maketitle

\section{Introduction and Notation}

We aim at giving an overview of the multipliers and convolutors of L. Schwartz' spaces of test functions and distributions in his treatise on distribution theory \cite{SCHW}. The continuity of the bilinear multiplication and convolution mappings in \cite{SCHW} (and every other book on distribution theory known to the author) is treated only sketchy and most of the considerations are limited to hypocontinuity, i.e. continuity if one of the components is restricted to a bounded subset. We completely describe the continuity properties of these mappings and give proofs mainly by working with seminorms. Besides, we point out some mistakes in the literature, in particular in L. Ehrenpreis' work \cite{EHR56} (Remark 5) and in Remark 3, since the continuity of separately continuous bilinear mappings often seems to be a source of error.
\par
For a part of the mappings under consideration, it was observed early that they are discontinuous and hence the very useful term of hypoconituity was brought in; this was done in the famous early paper of Dieudonn\'{e} and Schwartz \cite{DS49} and it can be seen quite easily that all of our multiplications and convolutions are hypocontinuous (cf. Section 3.2, p.~6). 
\par
But it is often also essential to know if a bilinear mapping is continuous or only hypocontinuous. As an example, we explain this point by L. Schwartz' proposition on the ``elementary'' convolution of vector-valued distributions \cite[Prop.~34, p.~151]{SCHW58}: 
Let $\mathcal{H}$, $\mathcal{K}$, $\mathcal{L}$ be three distribution spaces with certain properties, and $E$ and $F$ be separated locally convex spaces. If
\[\ast\colon: \mathcal{H}\times \mathcal{K}\rightarrow \mathcal{L}, (S,T)\mapsto S\ast T,\]
is hypocontinuous convolution then there exists a separately continuous convolution
\[\begin{smallmatrix}\ast\\\pi\end{smallmatrix}\colon \mathcal{H}(E)\times \mathcal{K}(F)\rightarrow \mathcal{L}(E\hat{\otimes}_\pi F),\]
fulfilling the consistency condition \[(S\otimes e)\begin{smallmatrix}\ast\\\pi\end{smallmatrix}(T\otimes f)=(S\ast T)\otimes (e\otimes f)\] for all $(S,T)\in\mathcal{H}\times\mathcal{K}$ and $(e,f)\in E\times F.$
\par
From this proposition, a vector-valued convolution emerges if a further bilinear mapping $b\colon E\times F\rightarrow G$ ($G$ a complete locally convex space) is given which can be continued to the completed $\pi$-tensor product as $\tilde{b}\colon E\hat{\otimes}_\pi F \rightarrow G$ and then yields the mapping $\begin{smallmatrix}\ast\\ b\end{smallmatrix} \colon \mathcal{H}(E)\times \mathcal{K}(E)\rightarrow \mathcal{L}(G)$ as the composition $\begin{smallmatrix}\ast\\ b\end{smallmatrix} =(\text{id}_{\mathcal{L}} \epsilon \tilde{b})\circ \begin{smallmatrix}\ast\\\pi\end{smallmatrix}$.
The key point is the possibility of continuation of $b$ from $E\times F$ to $E\hat{\otimes}_\pi F$ which is only possible for continuous $b$ but not for hypocontinuous $b$. As for instance the proposition does not furnish the existence of a vector-valued convolution
\[\begin{smallmatrix}\ast\\\ast\end{smallmatrix}\colon \mathcal{E}'(\mathcal{D'})\times\mathcal{D}'(\mathcal{E'})\rightarrow \mathcal{D}'(\mathcal{D'}).\] However it implies the existence (and the uniqueness) of the vector-valued convolution
\[\begin{smallmatrix}\ast\\\ast\end{smallmatrix}\colon \mathcal{E}'(\mathcal{E'})\times\mathcal{D}'(\mathcal{E'})\rightarrow \mathcal{D}'(\mathcal{D'}).\]
\\\par
If $E$ is a space of $\mathcal{C}^\infty$--functions on $\R^n$, we call $M(E)$ the space of multipliers of $E$, i.e. the largest space of $\mathcal{C}^\infty$--functions, such that the multiplication on $E\times M(E)\stackrel{\cdot}{\rightarrow}\mathcal{C}^\infty$ is well-defined and takes values in $E$. If $E$ is a distribution space, the space $M(E)$ is defined analogously. But in this case the multiplication with $\mathcal{C}^\infty$--functions is defined by transposition, such as the multiplication $\mathcal{D}'\times\mathcal{C^\infty}\stackrel{\cdot}{\rightarrow}\mathcal{D}'$ in \cite[Chap. 5, §1]{SCHW}, i.e. for $(f,T)\in\mathcal{C}^\infty\times\mathcal{D}'$ the distribution $fT$ is defined as $\langle \varphi,fT\rangle=\langle f\varphi,T\rangle$ for $\varphi\in\mathcal{D}$, where here and in the rest of the paper $\langle-,-\rangle$ denotes the evaluation mapping. For a function or distribution space $E$, we call $C(E)$ the space of convolutors of $E$, i.e. the largest space, such that the convolution $E\times C(E)\stackrel{\ast}{\rightarrow}\mathcal{D}'$ is well-defined and takes values in $E$. For convolvability of distributions, we refer to one of the (equivalent) conditions in \cite[p.~315]{ORT10}. In Section 2 we will discuss multipliers and convolutors for the function spaces
\par
$\mathcal{D}=\{\varphi\in\mathcal{C}^\infty(\R^n)\,;\,\supp\varphi\text{ compact}\}$ (test functions)
\par
$\mathcal{S}=\{\psi\in\mathcal{C}^\infty(\R^n)\,;\,\forall\alpha,\beta\in\N_0^n\colon x^\alpha\partial^\beta\psi(x)\in\mathcal{C}_0\}$ (rapidly decreasing functions)
\par
$\mathcal{D}_{L^p}=\{f\in\mathcal{C}^\infty(\R^n)\,;\,\forall\alpha\in\N_0^n\colon\partial^\alpha f\in L^p\}$ (Sobolev space $W^{\infty,p}$), $1\leq p <\infty$,
\par
$\bp\,\,=\{f\in\mathcal{C}^\infty(\R^n)\,;\,\forall\alpha\in\N_0^n\colon\partial^\alpha f\in \mathcal{C}_0\}$

\par
$\oc=\{f\in\mathcal{C}^\infty(\R^n)\,;\,\exists k\in\N_0\,\forall\alpha\in\N_0^n\colon (1+|x|^2)^{-k}\partial^\alpha f(x)\in\mathcal{C}_0\}$ (very slowly increasing functions)
\par
$\om=\{f\in\mathcal{C}^\infty(\R^n)\,;\,\forall\alpha\in\N_0^n\,\exists k\in\N_0 \colon(1+|x|^2)^{-k}\partial^\alpha f(x)\in\mathcal{C}_0\}$ (slowly increasing functions)
\par
$\mathcal{E}=\mathcal{C}^\infty(\R^n)$
\\
and of their duals, the distribution spaces
\par
$\mathcal{D}'$ (distributions)
\par
$\mathcal{S'}$ (tempered distributions)
\par
$\mathcal{D}'_{L^q}=\{T\in\mathcal{D}'\,;\,\exists m\in\N_0\colon T=\sum_{|\alpha|\leq m} \partial^\alpha f_\alpha\text{ with }f_\alpha\in L^q\}$, $1< q \leq \infty$,
\par
$\mathcal{D}'_{L^1}=\{T\in\mathcal{D}'\,;\,\exists m\in\N_0\colon T=\sum_{|\alpha|\leq m} \partial^\alpha f_\alpha\text{ with }f_\alpha\in L^1\}$
\par
$\oc'=\{T\in\mathcal{D}'\,;\,\forall k\in\N_0 \colon(1+|x|^2)^{k}T\in\mathcal{D}'_{L^\infty}\}$ (rapidly decreasing distributions)
\par
$\om'=\{T\in\mathcal{D}'\,;\,\exists m\,\forall k\colon(1+|x|^2)^{k}T=\sum_{|\alpha|\leq m} \partial^\alpha f_\alpha \text{ with } f_\alpha\in L^\infty\}$ (very rapidly decreasing distributions)
\par
$\mathcal{E'}$ (distributions with compact support),
\\ where $\mathcal{C}^\infty(\R^n)$ is the space of infinitely differentiable functions on $\R^n$ (with complex values), $\mathcal{C}_0$ denotes the space of continuous functions on $\R^n$ tending to 0 at infinity and $L^p$ are the usual Lebesgue spaces. We have the continuous inclusions (cf. \cite[p.~419]{SCHW})
\begin{large}
\begin{center}
\begin{tabular}{c c c c c c c c c c c c c}
$\mathcal{D}$ & $\subset$ &$\mathcal{S}$ & $\subset$ &$\mathcal{D}_{L^p}$& $\subset$ &$\bp$ & $\subset$ &$\oc$ & $\subset$ &$\om$ & $\subset$ &$\mathcal{E}$ \\
$\cap$ &&&&&&&&&&&& $\cap$ \\
$\mathcal{E'}$& $\subset$&$\om'$& $\subset$&$\oc'$& $\subset$&$\mathcal{D}'_{L^1}$&$\subset$& $\mathcal{D}'_{L^q}$& $\subset$&$\mathcal{S'}$& $\subset$&$\mathcal{D}'$ \\
\end{tabular}
\end{center}
\end{large}

In Section 3.1, we will give an overview of the seminorms defining the topologies of the different function and distribution spaces in order to investigate the continuity of the existing multiplications and convolutions in Section 3.2. For the different terms of continuity of bilinear mappings, we refer to \cite[pp.~III.28--III.32]{BOU}. 

\section{Multipliers and convolutors of function and distribution spaces}

From the following table, one can read off the multipliers and convolutors of the spaces listed above. The symbols 'o' indicate continuous and the symbols 'x' indicate discontinuous multiplication (or convolution) mappings $E\times M(E)\stackrel{\cdot}{\rightarrow}E$ (or $E\times C(E)\stackrel{\ast}{\rightarrow}E$), respectively. The number following thereafter is the number of the proposition, where the continuity or discontinuity is proved.

\begin{remark}
In some cases the space $C(E)$ and the c-dual $E^*$ of $E$, i.e. the largest space, such that the convolution on $E\times E^*\rightarrow\mathcal{D}'$ is well-defined, coincide. Hence for some parts of the table see also Theorem 5 in \cite[p.~22]{YO58}, where the c-duals of the function and distribution spaces of L. Schwartz are determined.
\end{remark}\pagebreak
\begin{large}
\begin{center}
multiplier-convolutor-table \\\vspace{0.3cm}
\begin{tabular}{l |c| c}
     $E$      & $M(E)\hfill$ \begin{scriptsize}\end{scriptsize}        &    $C(E)\hfill$ \begin{scriptsize}\end{scriptsize}       \\\hline \hline
$\mathcal{D}$ &       $\mathcal{E}\hfill$ \begin{scriptsize}x  $_1$\end{scriptsize}      &      $\mathcal{E'}\hfill$ \begin{scriptsize}x  $_2$\end{scriptsize} \\\hline
$\mathcal{S} $  &       $\om\hfill$ \begin{scriptsize}x  $_2$\end{scriptsize}     &      $\oc'\hfill$ \begin{scriptsize}x  $_2$\end{scriptsize}  \\ \hline
$\mathcal{D}_{L^p}  $ &       $\mathcal{D}_{L^\infty}\hfill$ \begin{scriptsize}o  $_3$\end{scriptsize}  & $\mathcal{D}'_{L^1}\hfill$ \begin{scriptsize}x  $_2$\end{scriptsize}   \\\hline

 $\bp$&       $\mathcal{D}_{L^\infty}\hfill$ \begin{scriptsize}o  $_3$\end{scriptsize}  & $\mathcal{D}'_{L^1}\hfill$ \begin{scriptsize}x  $_2$\end{scriptsize} \\\hline
$\oc$&$\oc  \hfill$ \begin{scriptsize}x  $_4$\end{scriptsize}  &   $\oc'  \hfill$ \begin{scriptsize}x  $_2$\end{scriptsize}    \\\hline
$\om$  &        $\om\hfill$ \begin{scriptsize}o  $_5$\end{scriptsize}       &      $\om'\hfill$ \begin{scriptsize}x  $_{2}$\end{scriptsize} \\\hline
$\mathcal{E}$   &       $\mathcal{E}\hfill$ \begin{scriptsize}o  $_3$\end{scriptsize}   &    $\mathcal{E'}\hfill$ \begin{scriptsize}x  $_2$\end{scriptsize} \\\hline\hline

$\mathcal{E'} $ &       $\mathcal{E} \hfill$ \begin{scriptsize}x  $_{6}$\end{scriptsize} &   $\mathcal{E'} \hfill$ \begin{scriptsize}o  $_3$\end{scriptsize}   \\\hline
$\om'$  &     $\om\hfill$  \begin{scriptsize}x  $_{6}$\end{scriptsize}    &     $\om'\hfill$ \begin{scriptsize}x  $_{4}$\end{scriptsize}\\\hline
$\oc' $  &     $\oc\hfill$  \begin{scriptsize}x  $_{6}$\end{scriptsize}     &     $\oc'\hfill$ \begin{scriptsize}o  $_{5}$\end{scriptsize}\\\hline
$\mathcal{D}'_{L^1}  $ &       $\mathcal{D}_{L^\infty}\hfill$ \begin{scriptsize}x  $_{6}$\end{scriptsize} & $\mathcal{D}'_{L^1}\hfill$ \begin{scriptsize}o  $_{3}$ \end{scriptsize}\\\hline
$\mathcal{D}'_{L^q}  $ &       $\mathcal{D}_{L^\infty}\hfill$ \begin{scriptsize}x  $_{6}$\end{scriptsize} & $\mathcal{D}'_{L^1}\hfill$ \begin{scriptsize}o  $_{3}$ \end{scriptsize}\\\hline
$\mathcal{S}' $  &$\om\hfill$ \begin{scriptsize}x  $_{6}$\end{scriptsize}&  $\oc' \hfill$ \begin{scriptsize}x  $_{6}$\end{scriptsize}  \\\hline
$\mathcal{D}'$&$\mathcal{E}\hfill$ \begin{scriptsize}x  $_{6}$\end{scriptsize}&$\mathcal{E'}\hfill$ \begin{scriptsize}x  $_{7}$ \end{scriptsize}\\
\end{tabular}
\end{center}
\end{large}
\vspace{10pt}
The multipliers and convolutors of the spaces $\mathcal{D}$, $\mathcal{S}$, $\mathcal{E}$ and of their duals are widely known and largely contained in \cite{SCHW}.  Hölder's inequality tells us that the multiplication on ${L^p}\times{L^{s}}$ takes values in ${L^r}$ if $\tfrac1p+\tfrac{1}{s}=\tfrac1r$. Young's inequality asserts that the convolution on ${L^p}\times{L^{s}}$ takes values in ${L^r}$ if $\tfrac1p+\tfrac1s=1+\tfrac1r$. In both cases we want $p=r$, which leads us to the multipliers and convolutors of the spaces $\mathcal{D}_{L^p} $ and $\mathcal{D}'_{L^q}$, cf. \cite[pp.~203f]{SCHW}. That the space of multipliers of $\bp$ is $\mathcal{D}_{L^\infty}$ is obvious and that the convolution on $\bp\times \mathcal{D}'_{L^1}$ takes values in $\bp$ is due to the well-definedness of the convolution $\mathcal{C}_0\times L^1\stackrel{\ast}{\rightarrow}\mathcal{C}_0$ (dominated convergence). Note that if $g\in\mathcal{D}_{L^p}$, $1\leq p <\infty$, (or $g\in\,\bp$) and $T\in\mathcal{D}'_{L^q}=(\mathcal{D}_{L^p})'$, $1/p+1/q=1$, (or $T\in\mathcal{D}'_{L^1}=(\bp)'$) is of the form $T=\sum_{|\alpha|\leq m}  \partial^\alpha f_\alpha$, $f_\alpha\in L^q$, the evaluation $\langle  g,T \rangle$ is defined as
$$\langle  g,T \rangle=\sum_{|\alpha|\leq m} (-1)^{|\alpha|}\langle \partial^\alpha g, f_\alpha\rangle=\sum_{|\alpha|\leq m} (-1)^{|\alpha|}\int_{\R^n}\partial^\alpha g(x) f_\alpha(x)~\dx x,$$
according to the definition of differentiation of distributions. 
\par The strong dual of $\mathcal{D}_{L^\infty}$, denoted $\mathcal{B}'$ by L. Schwartz, is not a space of distributions \cite[p.~200]{SCHW}, so we excluded this space from our considerations. To obtain symmetry regarding duality in the table, we also left out the space $\mathcal{D}_{L^\infty}$ itself, which would be located between $\bp$ and $\oc$. The multipliers and convolutors of $\mathcal{D}_{L^\infty}$ are the same as that of the spaces $\mathcal{D}_{L^p}$, $1\leq p<\infty$, and $\bp$.
\par 
We will now proof $M(\oc)=\oc$ and $M(\om)=\om$. Function spaces $E$, which contain the constant functions, fulfill 
$$M(E)=\{1\}\cdot M(E)\subset E\cdot M(E)= E.$$
 Thus, since $1\in\oc$ and $1\in\om$, $M(\om)$ has to be a subspace of $\om$ and $M(\oc)$ a subspace of $\oc$. Note that neither $\mathcal{D}$, $\mathcal{S}$, $\mathcal{D}_{L^p}$ (for $1\leq p<\infty$) nor $\bp$ contain the function $x\mapsto 1$ and that they are not supersets of their multipliers.
 
 \begin{remark} The statement for the space $\om$ of the following lemma appears also in \cite[p.~88]{EKO10}, Prop. 6.10, and \cite[p.~14]{GV92}, Prop. 2 (i). Note that part (ii) of Proposition~2 in \cite{GV92} contains a misleading sentence, since it states that the operation of multiplication $\om\times \mathcal{S}\stackrel{\cdot}{\rightarrow} \mathcal{S}$ is well-defined (which is correct) and that ``this operator is continuous''. In Section 3.2 we will see that this multiplication mapping is separately continuous (and even hypocontinuous) but not jointly continuous. Therefore, what we have is that every $f\in\om$ defines a continuous multiplication operator $\nolinebreak{\mathcal{S}\rightarrow \mathcal{S}, \psi\mapsto f\psi}$. \end{remark}

\begin{lemma}
$(\oc,\cdot)$ and $(\om,\cdot)$ are multiplication algebras.
\end{lemma}
\begin{proof}
$\oc$ and $\om$ are clearly vector spaces. Let now $f,g\in\oc$. There are $k_1,k_2\in\N_0$ such that for all $\alpha\in\N_0^n$ the functions $(1+|x|^2)^{-k_1}\partial^\alpha f(x)$ and $(1+|x|^2)^{-k_2}\partial^\alpha g(x)$ are in $\mathcal{C}_0$. Thus, if we set $k=k_1+k_2$ we get $(1+|x|^2)^{-k}\partial^\alpha (f\cdot g)\in\mathcal{C}_0$ for all $\alpha\in\N_0^n$. 
\par 
Let $f,g\in\om$ and $\alpha\in\N_0^n$. We choose $k_1$ and $k_2$ with  $(1+|x|^2)^{-k_1}\partial^\alpha f(x)\in\mathcal{C}_0$ and $(1+|x|^2)^{-k_2}\partial^\alpha g(x)\in\mathcal{C}_0$, $\beta\leq\alpha$. Setting again $k=k_1+k_2$ we get $(1+|x|^2)^{-k}\partial^\beta (f\cdot g)\in\mathcal{C}_0$, $\beta\leq\alpha$. 
\end{proof}
Since the multiplication of distributions with infinitely differentiable functions is defined by transposition, the multipliers of the distribution spaces $\om'$ and $\oc'$ are the same as for their preduals $\om$ and $\oc$, respectively.
Hence we also have $M(\oc')=\oc$ and $M(\om')=\om$. \par If we denote by $\mathcal{F}$ the Fourier transform, we have $\mathcal{F}(\oc)=\om'$ and $\mathcal{F}(\om)=\oc'$, see Th\'{e}or\`{e}me XV in \cite[p.~268]{SCHW}. The Convolution Theorem (or Exchange Theorem/Fourier Exchange Formula) states that the equation $\mathcal{F}(S \ast T)=\mathcal{F}S \cdot \mathcal{F}T$ holds for pairs $(S,T)$ in $\oc'\times\mathcal{S'}$, see Theorem 3 in \cite[p.~424]{HOR}. This equation also holds for pairs in $\oc'\times\oc$, $\om'\times\om$, $\om'\times\om'$ and $\oc'\times\oc'$, since these spaces are contained in $\oc'\times\mathcal{S'}$. Hence $M(\oc)=\oc$ and $M(\om)=\om$ yield $C(\om')=\om'$ and $C(\oc')=\oc'$, respectively. Furthermore, we have $C(\om)=\om'$ and $C(\oc)=\oc'$, since $M(\oc')=\oc$ and $M(\om')=\om$, respectively. \par
Note that the identity $C(\mathcal{S})=C(\oc)=C(\mathcal{S'})=C(\oc')=\oc'$ is also contained in \cite[p.~53]{GV92}.

\begin{lemma}$(\oc',\ast)$ and $(\om',\ast)$ are convolution algebras.\end{lemma}


\section{Continuity of the multiplications and convolutions}

\subsection{Seminorms defining the topologies of the function and distribution spaces}

In order to survey the continuity of mappings between these function and distribution spaces, we need to know the seminorms defining their topology. A system of defining seminorms of $\mathcal{D}$ is formed by the seminorms
$$p_{\boldsymbol{m},\boldsymbol{\varepsilon}}(\varphi)=\sup_{\nu\in\N_0}\left(\sup_{|x|\geq\nu, |\alpha|\leq m_\nu}\frac{1}{\varepsilon_\nu}|\partial^\alpha\varphi(x)|\right),\,\varphi\in\mathcal{D},$$
 where $\boldsymbol{m}=(m_k)_{k\in\N_0}$ is a sequence of natural numbers tending to infinity and $\boldsymbol{\varepsilon}=(\varepsilon_k)_{k\in\N_0}$ is a sequence of real numbers in $(0,\infty)$ tending to 0 (see Chap. 3, §1, p.~65 in \cite{SCHW}). Note that every $p_{\boldsymbol{m},\boldsymbol{\varepsilon}}$ is even a norm, since it dominates $f\mapsto\Vert f\Vert_\infty=\sup_{x\in\R^n}|f(x)|$ (set $\nu=0$ and $\alpha=0$), and that $p_{\boldsymbol{m},\boldsymbol{\varepsilon}}(\varphi)=\sup_{|\alpha|\leq m_0}\frac{1}{\varepsilon_0}\Vert\partial^\alpha\varphi\Vert_\infty$ if $\supp\varphi\subset\{x\in\R^n\,;\, |x|\leq 1\}$. 
\par
The seminorms of $\mathcal{S}$ are 
$$p_{m,\beta}(\varphi)=\sup_{x\in\R^n,|\alpha|\leq m}|x^\beta \partial^\alpha \varphi(x)|,\,\varphi\in\mathcal{S},$$ where $m\in\N_0$ and $\beta\in\N_{0}^{n}$.
All these mappings are even norms, and if we set $m=0$ and $\beta=0$ we obtain the $\Vert\cdot\Vert_\infty$-norm. 
\par
For $m\in\N_0$
$$p_m(f)=\sup_{|\alpha|\leq m}\Vert \partial^\alpha f\Vert_p$$
are the defining seminorms (which are again even norms) of $\mathcal{D}_{L^p}$. For $p=\infty$ these are also the norms of $\bp$ since this space carries the subspace topology relative to $\mathcal{D}_{L^\infty}$. 
The set $\{ p_{m,\psi}\,;\,m\in\N_0,\psi\in\mathcal{S}\} $, where 
$$p_{m,\psi}(f)=\sup_{|\alpha|\leq m}\Vert\psi\cdot \partial^\alpha f\Vert_\infty$$
for $f\in\om$, is a system of norms defining the topology of $\om$.
\par
For a compact subset $K$ of $\R^n$ and $m\in\N_0$ the seminorms
$$p_{m,K}(f)=\sup_{x\in K,|\alpha|\leq m}|\partial^\alpha f(x)|$$
form a system, which defines the topology of $\mathcal{E}$. The space $\mathcal{E}$ is the only function space considered, without continuous norm, i.e. there is no neighbourhood of zero, that doesn't contain any straight line.
\par
The spaces $\mathcal{S}$, $\mathcal{D}_{L^p}$, $\bp$ and $\mathcal{E}$ are metrizable since their topology can be defined by a countable family of seminorms. According to \cite[p.~442]{HOR} they are also complete and, hence, Fr\'{e}chet spaces. Their duals are (DF)-spaces (see Def. 1 in \cite[p.~63]{GRO54} or \cite[p.~257]{JAR} for the definition of (DF)-spaces).
\par 
The continuous seminorms of the distribution spaces $\mathcal{D}', \mathcal{S'}, \mathcal{D}'_{L^q}, \oc'$, $\om'$ and $\mathcal{E'}$ are defined as the supremum of the evaluation on bounded subsets of their predual. Seminorms on $\mathcal{E'}$, for example, are defined as $p_B(T)=\sup_{g\in B}|\langle g,T \rangle|$, where $B$ is a bounded subset of $\mathcal{E}$.
 \par
Since $\mathcal{D}$ is the regular inductive limit of the spaces $\mathcal{D}_K=\{\varphi\in\mathcal{E}\,;\,\supp\varphi\subset K\}$, where $K$ is a compact subset of $\R^n$, a subset $B$ of $\mathcal{D}$ is bounded if it is bounded in $\mathcal{E}$ and it exists a compact subset of $\R^n$, that contains $\supp\varphi$ for every $\varphi\in B$, see Th\'{e}or\`{e}me IV in \cite[p.~69]{SCHW}.
\\\par
We  will see that the continuity of several multiplications is equivalent (via Fourier transform) to the continuity of a corresponding convolution, since the Fourier transform yields isomorphisms $\mathcal{S}\rightarrow \mathcal{S}$, $\mathcal{S'}\rightarrow \mathcal{S'}$, $\om\rightarrow \oc'$ and $\oc\rightarrow \om'$---see Th. XII, p.~249, and Th. XV, p.~268, in \cite{SCHW}---and since the Convolution Theorem mentioned after Lemma 1 holds.

\subsection{Investigation of continuity}

Every occurring multiplication $\cdot\colon E\times M(E)\rightarrow E$ and convolution $\ast \colon E\times C(E)\rightarrow E$ is separately continuous (use, for example, a closed graph theorem), and since all of these function and distribution spaces are barrelled, these mappings are even hypocontinuous (see Theorem 2 in \cite[p.~360]{HOR}). Convergent sequences are bounded and hypocontinuous maps are continuous if one component is restricted to a bounded set and, thus, hypocontinuous maps are in particular sequentially continuous. \par
Since all of the multiplications and convolutions are separately continuous, the algebras $(\mathcal{D}_{L^\infty},\cdot)$, $(\oc,\cdot)$, $(\om,\cdot)$, $(\mathcal{E},\cdot)$, $(\mathcal{E'},\ast)$, $(\om',\ast)$, $(\oc',\ast)$, $(\mathcal{D}'_{L^1},\ast)$ are topological algebras in the sense of \cite[p.~4]{MAL86} and \cite[p.~6]{FRA05}. In the rest of this article, the notion of a topological algebra will be used in the sense of \cite[p.~177]{BNS77} and \cite[Chap. 20]{VDW91}, i.e. it denotes an algebra with (jointly) continuous operations. We will now justify the symbols 'x' and 'o' in the multiplier-convolutor-table.

In most of the following proofs, we will use the seminorms of Section 3.1. Let us recall that a bilinear mapping $b\colon E\times F\rightarrow G$ between locally convex topological vector spaces $E,F,G$ is continuous if and only if for every continuous seminorm $p_1$ on $G$ there exist continuous seminorms $p_2$ and $p_3$ on $E$ and $F$, respectively, such that the inequality $$p_1(b(v,w))\leq p_2(v)p_3(w)$$ holds for every pair $(v,w)\in E\times F$.

\begin{proposition} 
The mapping $\mathcal{D}\times\mathcal{E}\stackrel{\cdot}{\rightarrow}\mathcal{D} $ is discontinuous.
\end{proposition}

\begin{proof}
We have to show that there is a seminorm $p$ on $\mathcal{D}$ such that for every pair of seminorms $\tilde{p}$ and $p_{m,K}$ on $\mathcal{D}$ and $\mathcal{E}$, respectively, there exists a pair $(\varphi,f)\in\mathcal{D}\times\mathcal{E}$ with $$p(\varphi\cdot f)>\tilde{p}(\varphi)p_{m,K}(f).$$ We can choose an arbitrary norm $p$ on $\mathcal{D}$. If $p_{m,K}$ is defined by the compact set $K$ and $m\in\N_0$, we can find $0 \neq f\in\mathcal{E}$ with $p_{m,K}(f)=0$. If we take $\varphi\in\mathcal{D}$ with $\varphi\cdot f \neq 0$ we have $p(\varphi\cdot f)>0=\tilde{p}(\varphi)p_{m,K}(f)$. 
\end{proof}

In fact, if $\mathcal{E}$ is replaced by $\mathcal{D}$, we obtain the continuous multiplication $\mathcal{D}\times\mathcal{D}\stackrel{\cdot}{\rightarrow}\mathcal{D}$ \cite{HSTH01}, making $(\mathcal{D},\cdot)$ a topological algebra.

\begin{proposition}
The regularization mappings $\mathcal{D}\times\mathcal{E'}\stackrel{\ast}{\rightarrow}\mathcal{D}$, $\mathcal{S}\times\oc'\stackrel{\ast}{\rightarrow}\mathcal{S}$, $\mathcal{D}_{L^p}\times\mathcal{D}'_{L^1}\stackrel{\ast}{\rightarrow}\mathcal{D}_{L^p}$, $\bp\times\mathcal{D}'_{L^1}\stackrel{\ast}{\rightarrow}\bp$, $\oc\times\oc'\stackrel{\ast}{\rightarrow}\oc$, $\om\times\om'\stackrel{\ast}{\rightarrow}\om$ and $\mathcal{E}\times\mathcal{E'}\stackrel{\ast}{\rightarrow}\mathcal{E}$ are discontinuous. 
\end{proposition}
\begin{proof}
We will prove the discontinuity of all these regularizations at once. Since $\mathcal{D}\times\mathcal{E'}$ is continuously embedded in the domains of these mappings and their target spaces are embedded in $\mathcal{E}$ it suffices to prove that $\mathcal{D}\times\mathcal{E'}\stackrel{\ast}{\rightarrow}\mathcal{E}$ is discontinuous.  \par 
Let us assume that this convolution is continuous. Since $f\mapsto \sup_{|x|\leq 1}|f(x)|$ is a continuous seminorm on $\mathcal{E}$, for every seminorm $p_{\boldsymbol{m},\boldsymbol{\varepsilon}}$ on $\mathcal{D}$, every bounded subset $B$ of $\mathcal{E}$--defining a seminorm $p_B$ on $\mathcal{E'}$--
$$\sup_{|x|\leq 1}|(f\ast T)(x)|\leq p_{\boldsymbol{m},\boldsymbol{\varepsilon}}(f)p_B(T)$$ holds for every pair $(f,T)\in\mathcal{D}\times\mathcal{E'}$.
\par 
To attain a contradiction we choose $\gamma\in\N_{0}^{n}$ with $|\gamma|=m_0+1$, $T=\partial^\gamma\delta$, $f\in\mathcal{D}\backslash\{0\}$ with $\supp f\subset\{x\in\R^n\,;\, |x|\leq 1\}$ and $c=(p_B(T)+1)p_{\boldsymbol{m},\boldsymbol{\varepsilon}}(f)/\sup_{|x|\leq 1}|\partial^\gamma {f}(x)|$. 
By assumption we have $c>1$ and hence the support of the function $\tilde{f}\in\mathcal{D}$ defined by $\tilde{f}(x)=f(cx)$ is also contained in $\{x\,;\, |x|\leq 1\}$. 
Furthermore it holds $$\sup_{|x|\leq 1}|(\tilde{f}\ast T)(x)|=\sup_{|x|\leq 1}|\partial^\gamma \tilde{f}(x)|=c^{m_0+1}\sup_{|x|\leq 1}|\partial^\gamma {f}(x)|$$ and 
$$p_{\boldsymbol{m},\boldsymbol{\varepsilon}}(\tilde{f})=\sup_{|\alpha|\leq m_0}\tfrac{1}{\varepsilon_0}\Vert \partial^\alpha\tilde{f}\Vert_\infty\leq c^{m_0}\sup_{|\alpha|\leq m_0}\tfrac{1}{\varepsilon_0}\Vert \partial^\alpha f \Vert_\infty=c^{m_0} p_{\boldsymbol{m},\boldsymbol{\varepsilon}}(f).$$
But this yields $$\sup_{|x|\leq 1}|(\tilde{f}\ast T)(x)|\geq (p_B(T)+1) p_{\boldsymbol{m},\boldsymbol{\varepsilon}}(\tilde{f})> p_B(T) p_{\boldsymbol{m},\boldsymbol{\varepsilon}}(\tilde{f}).$$
\end{proof}

\hspace{-0.51cm}\textit{Alternative proof for the regularization mappings on $\bp$, $\oc$, $\om$ or $\mathcal{E}$.}
Let $E$ be one of the spaces $\bp$, $\oc$, $\om$ or $\mathcal{E}$ and assume that the convolution $E\times E'\stackrel{\ast}{\rightarrow} E$ is continuous. For $f\in E$ and $T\in E'$ we can write $(f\ast T)(x)$ as $\langle f(x-y), T(y)\rangle$. If we apply $\delta_0(x)$ we get $\langle f(-y), T(y)\rangle=\langle \check{f}, T\rangle$, where $\check{f}(x):=f(-x)$. The Dirac delta is in $E'$, thus the evaluation $(f,T)\mapsto \langle f,T\rangle=T(f)$ on $E\times E'$ is continuous, as it is the composition of the continuous mappings $\delta\colon E\rightarrow \C$, $\ast\colon E\times E'\rightarrow E$ and $E\rightarrow E,f\mapsto \check{f}$. But the evaluation on $E \times E'$ is only continuous if $E$ is a normed space (see \cite{HOR}, p.~359), which is not the case for the spaces $\bp$, $\oc$, $\om$ and $\mathcal{E}$.

$\hfill\Box$
\\\par
By Fourier transform, the discontinuity of the convolutions $\oc\times\oc'\stackrel{\ast}{\rightarrow}\oc$ and $\om\times\om'\stackrel{\ast}{\rightarrow}\om$ is equivalent to the discontinuity of the multiplications $\om'\times\om\stackrel{\cdot}{\rightarrow}\om'$ and $\oc'\times\oc\stackrel{\cdot}{\rightarrow}\oc'$, respectively. But we will also prove the discontinuity of these multiplications among others in Proposition 6. \par
The assertion of Proposition 2 obviously also holds for the regularization $\mathcal{D}_{L^\infty}\times \mathcal{D}'_{L^1}\stackrel{\ast}{\rightarrow}\mathcal{D}_{L^\infty}$. \par Note also that if we replace $\mathcal{E'}$ by $\mathcal{D}$ in the convolution $\mathcal{D}\times\mathcal{E'}\stackrel{\ast}{\rightarrow}\mathcal{D}$, we obtain the continuous convolution $\mathcal{D}\times\mathcal{D}\stackrel{\ast}{\rightarrow}\mathcal{D}$ \cite{HSTH01}. Hence $(\mathcal{D},\ast)$ is a topological algebra too.

\begin{proposition}\label{prop:ff}
Let $1\leq p,q\leq\infty$. The multiplications $\mathcal{D}_{L^p}\times \mathcal{D}_{L^\infty}\stackrel{\cdot}{\rightarrow} \mathcal{D}_{L^p}$,\linebreak $\bp\times \mathcal{D}_{L^\infty}\stackrel{\cdot}{\rightarrow} \,\bp$ and $\mathcal{E}\times \mathcal{E}\stackrel{\cdot}{\rightarrow} \mathcal{E}$ and the convolutions $\mathcal{D}'_{L^q}\times \mathcal{D}'_{L^1}\stackrel{\ast}{\rightarrow} \mathcal{D}'_{L^q}$ and \linebreak $\mathcal{E'}\times \mathcal{E'}\stackrel{\ast}{\rightarrow} \mathcal{E'}$ are continuous, since they are bilinear, hypocontinuous maps, defined on the product of two \text{Fr\'{e}chet} spaces or (DF)-spaces (i.e. duals of Fr\'{e}chet spaces), respectively (see Th. 2 in \cite[p.~64]{GRO54}). In particular, the algebras $(\mathcal{D}_{L^\infty},\cdot)$, $(\mathcal{E},\cdot)$, $(\mathcal{E'},\ast)$ and $(\mathcal{D}'_{L^1},\ast)$ are even topological algebras.
\end{proposition}

Although not appearing in the multiplier-convolutor-table, the multiplications $\mathcal{S}\times\mathcal{S}\stackrel{\cdot}{\rightarrow}\mathcal{S}$, $\bp\times \bp\,\stackrel{\cdot}{\rightarrow} \,\bp\,$ and $\mathcal{D}_{L^p}\times \mathcal{D}_{L^p}\stackrel{\cdot}{\rightarrow} \mathcal{D}_{L^p}$ and the convolutions $\mathcal{S}\times\mathcal{S}\stackrel{\ast}{\rightarrow}\mathcal{S}$ and $\mathcal{D}_{L^1}\times\mathcal{D}_{L^1}\stackrel{\ast}{\rightarrow}\mathcal{D}_{L^1}$ are also well-defined. Furthermore by the same argument as that of Proposition 3, they are continuous, making $(\mathcal{S},\cdot)$, $(\bp,\cdot)$, $(\mathcal{D}_{L^p},\cdot)$, $(\mathcal{S},\ast)$ and $(\mathcal{D}_{L^1},\ast)$ topological algebras too. \\\par

The proof of the following proposition is due to Peter Wagner.

\begin{proposition}
The mappings $\oc\times\oc\stackrel{\cdot}{\rightarrow}\oc$ and $\om'\times\om'\stackrel{\ast}{\rightarrow}\om'$ are discontinuous.
\end{proposition}

\begin{proof}
By Fourier transform the discontinuity of these two mappings is equivalent. We will prove the discontinuity of the multiplication $\oc\times\oc\stackrel{\cdot}{\rightarrow}\oc$. 
\par
Since the topology of $\oc$ is defined by the inductive limit $\lim_{j \rightarrow}{(1+|x|^2)^{j}}\mathcal{D}_{L^\infty}$, a base of neighbourhoods of 0 in $\oc$ is $$\{V_{\mathbf{k}}\,;\, \mathbf{k}=(k_j)_{j\in\N}\in\N^{\N}\},$$ where $V_\mathbf{k}$ is the absolutely convex envelope of the union $\bigcup_{j\in \N}(1+|x|^2)^{j}U_{k_j}$ with $$U_l=\left\{ f\in\mathcal{D}_{L^\infty}\,;\, \sup_{|\alpha|\leq l}\Vert \partial^\alpha f\Vert_\infty\leq \frac1l \right\},$$ cf. \cite[Cor., p.~79]{RR}. \par
We assume that the mapping in question is continuous. Then for every $\mathbf{k}\in\N^{\N}$ there exists $\mathbf{m}\in\N^{\N}$ with $V_{\mathbf{m}}\cdot V_{\mathbf{m}}\subset V_{\mathbf{k}}$ and in particular $V_{\mathbf{m}}\cdot (1+|x|^2)U_{m_1}\subset V_{\mathbf{k}}$. Since the function $x\mapsto \tfrac1l\,$ is in $U_l$, $l\in\N$, this inclusion yields $(1+|x|^2)^{j+1}\frac{1}{m_j}U_{m_1}\subset V_{\mathbf{k}}$ for all $j\in\N$. Thus we have 
$$\forall \mathbf{k}\,\exists l\in\N\,\forall j\in\N\,\exists\varepsilon>0\colon (1+|x|^2)^{1+j}\varepsilon U_l\subset V_{\mathbf{k}}.$$ 
\par
Let now $\varphi\in\mathcal{D}$ with $\varphi=1$ in a neighbourhood of 0 and $f_r(x):= \text{e}^{\text{i}|x|^2}\varphi(x/r)\in\mathcal{D}\subset\oc$ for $r\in\N$ and $x\in\R^n$. We show that the sequence $(f_r)_{r\in\N}$ is a Cauchy sequence in $\oc$ and, thus, also convergent, since $\oc$ is complete \cite[Chap. 2, Th. 16, p. 131]{GRO}. But this sequence converges to  $\rm e^{\rm i|x|^2}$ in $\mathcal{D}'$, and since the embedding $\oc\hookrightarrow \mathcal{D}'$ is continuous and $\rm e^{\rm i|x|^2}\notin \oc$ this will complete the proof. \par
To prove that $(f_r)_r$ is a Cauchy sequence we have to show that 
$$\forall\mathbf{k}\,\exists N\in\N\,\forall r,s\geq N\colon f_r-f_s\in V_{\mathbf{k}}$$
 or by our previous considerations it suffices to show
\begin{multline}
\forall l\in\N\,\exists j\in\N\,\forall\varepsilon>0\,\exists N\in N \,\forall r,s\geq N \colon
f_r-f_s\in(1+|x|^2)^{j+1}\varepsilon U_l\,, \\
i.e., \sup_{|\alpha|\leq l}\left|\left|\partial^\alpha \left((f_r(x)-f_s(x))(1+|x|^2)^{-j-1}\right)\right|\right|_\infty\leq \frac{\varepsilon}{l}.\end
{multline}
We choose $j=l$ and get 
\begin{multline}
\partial^\alpha \left(f_r(x)(1+|x|^2)^{-j-1}\right)=\\
\sum_{\beta \leq \alpha}\sum_{\gamma\leq \beta} \binom{\alpha}{\beta}\binom{\beta}{\gamma}\partial^\gamma \rm e^{\rm i|x|^2}\partial^{\beta-\gamma}(\varphi(x/r))\partial^{\alpha-\beta}(1+|x|^2)^{-l-1}=\\
\rm e^{\rm i|x|^2}\sum_{\beta \leq \alpha,\gamma\leq \beta}\frac{P_{\beta,\gamma}(x)}{(1+|x|^2)^{l+|\alpha-\beta|}}\cdot\frac{\partial^{\beta-\gamma}(\varphi(x/r))}{1+|x|^2},
\end{multline} where $P_{\beta,\gamma}$ is a polynomial of degree less than or equal to $|\gamma|+|\alpha-\beta|$. The function $P_{\beta,\gamma}(x)(1+|x|^2)^{-l-|\alpha-\beta|}$ is in $L^\infty$. Furthermore for every $\alpha\in\N^{n}_{0}$ we have 
$$\forall \varepsilon\,\exists N\in\N \,\forall r,s\geq N\colon \Vert \partial^\alpha(\varphi(x/r)-\varphi(x/s))(1+|x|^2)^{-1}  \Vert_\infty<\varepsilon,$$
 which implies (1).
\end{proof} 

\begin{proposition}
The mappings $\om\times\om\stackrel{\cdot}{\rightarrow}\om$ and $\oc'\times\oc'\stackrel{\ast}{\rightarrow}\oc'$ are continuous, i.e. $(\om,\cdot)$ and $(\oc',\ast)$ are even topological algebras.\end{proposition}

\begin{proof}
By Fourier transform the continuity of the two mappings is equivalent. We will prove the continuity of the multiplication $\om\times\om\stackrel{\cdot}{\rightarrow}\om$. \par
For a given seminorm $p_{k,\psi}$---defined by a $k\in\N_0$ and an $\mathcal{S}$--function $\psi$---we search seminorms $p_{l,\psi_1}$ and $p_{m,\psi_2}$ on $\om$ with $p_{k,\psi}(fg)\leq p_{l,\psi_1}(f) p_{m,\psi_2}(g)$ for all $f,g\in\om$. But since every $\mathcal{S}$--function can be written as the convolution of two $\mathcal{S}$--functions (see Lemma 2 in \cite[pp.~529f]{MIY}) and via Fourier transformation this is equivalent to the fact that every $\mathcal{S}$--function is the product of two $\mathcal{S}$--functions, we can choose $\psi_1,\psi_2\in\mathcal{S}$ such that $\psi=\psi_1\cdot\psi_2$. Using the product formula we have for $f,g\in\om$ $$p_{k,\psi}(fg)=\sup_{\alpha\leq k}\|\psi\partial^\alpha (f\cdot g)\|_\infty\leq C(k) \sup_{\alpha\leq k}\|\psi_1\partial^\alpha f\|_\infty\cdot \sup_{\beta\leq k}\|\psi_2\partial^\beta g\|_\infty.$$
\end{proof} 

\begin{remark} The assertion of Proposition 5 is also stated in \cite[p.~248]{SCHW}.
The continuity of the convolutions $\oc'\times\oc'\stackrel{\ast}{\rightarrow} \oc'$ and $\mathcal{E'}\times\mathcal{E'}\stackrel{\ast}{\rightarrow}\mathcal{E'}$ is also stated, but not proved, in Example 5 \cite[pp.~211f]{SHI63}. Note that this Example 5 contains a mistake. Namely, it is stated that the convolution $\mathcal{D}\times\mathcal{E}\stackrel{\ast}{\rightarrow}\mathcal{E}$ is continuous, which is not the case (cf. \cite[Prop.~44, p.~70]{BARDISS}): For given compact sets $K$ and $\tilde{K}$ simply take $0\neq f\in\mathcal{E}$ with $f|_{K}=0$ and $\varphi\in\mathcal{D}$ with $(f\ast\varphi)|_{\tilde{K}}\neq 0$. Hence $p_{l,\tilde{K}}(f\varphi)>0=p_{m,K}(f)p(\varphi)$ for any continuous norm $p$ on $\mathcal{D}$ and $l,m\in\N_0$.
\end{remark}
\begin{proposition}
Let $1\leq q\leq\infty$. The multiplications $\mathcal{E'}\times\mathcal{E}\stackrel{\cdot}{\rightarrow}\mathcal{E'}$, $\om'\times\om\stackrel{\cdot}{\rightarrow}\om'$, $\oc'\times\oc\stackrel{\cdot}{\rightarrow}\oc'$, $\mathcal{D}'_{L^q}\times\mathcal{D}_{L^\infty}\stackrel{\cdot}{\rightarrow}\mathcal{D}'_{L^q}$, $\mathcal{S'}\times\om\stackrel{\cdot}{\rightarrow}\mathcal{S'}$ and $\mathcal{D}'\times\mathcal{E}\stackrel{\cdot}{\rightarrow}\mathcal{D}'$ and the convolution $\mathcal{S'}\times\oc'\stackrel{\ast}{\rightarrow}\mathcal{S'}$ are discontinuous.
\end{proposition}
\begin{proof}
The discontinuity of the mappings $\mathcal{S'}\times\om\stackrel{\cdot}{\rightarrow}\mathcal{S'}$ and $\mathcal{S'}\times\oc'\stackrel{\ast}{\rightarrow}\mathcal{S'}$ is equivalent by Fourier transform. To proof the discontinuity of the multiplications in question all at once, we take the largest target space $\mathcal{D}'$ and as domain we take $\mathcal{E'}\times\mathcal{D}_{L^\infty}$. Since $\mathcal{E'}\times \mathcal{D}_{L^\infty}$ is continuously embedded in the domains of the considered multiplications, the discontinuity of these multiplications is proved if we show that the mapping $\mathcal{E'}\times\mathcal{D}_{L^\infty}\stackrel{\cdot}{\rightarrow}\mathcal{D}'$ is discontinuous. \par
We take $\varphi\in\mathcal{D}$ with $\varphi=1$ in a neighbourhood of 0. Then $T\mapsto |\langle \varphi,T\rangle |$ is a continuous seminorm on $\mathcal{D}'$. Let now $m\in\N_0$, defining the seminorm $p_{m}$ on $\mathcal{D}_{L^\infty}$, and $B$ be a bounded subset of $\mathcal{E}$, defining the seminorm $p_B$ on $\mathcal{E'}$. We set $\gamma=(m+1,0\ldots,0)$, $T=\partial^\gamma\delta\in\mathcal{E'}$ and $c=p_B(T)+1$. Furthermore we define $f(x)=\rm e^{\rm i c x_1}$ for $x=(x_1,\ldots,x_n)\in\R^n$. Then $f\in\mathcal{D}_{L^\infty}$ and $|\langle \varphi,fT \rangle|=|\partial^\gamma f (0)|=c^{m+1}$. Combining this with $p_{m}(f)=\sup_{|\alpha|\leq m}\Vert \partial^\alpha f\Vert_\infty=  c^m$ yields
\begin{gather*}|\langle \varphi,fT \rangle|=c^{m+1}>p_B(T) c^m= p_B(T)p_{m}(f).\end{gather*} 
\end{proof}

\begin{remark} The discontinuity of the multiplications $\mathcal{D}'\times\mathcal{E}\stackrel{\cdot}{\rightarrow}\mathcal{D}'$ and $\mathcal{S'}\times\om\stackrel{\cdot}{\rightarrow}\mathcal{S'}$ is the content of \cite{KM81}, where the discontinuity of the two mappings is proved separately.
\end{remark}

\begin{proposition}
The mapping $\mathcal{D}'\times\mathcal{E'}\stackrel{\ast}{\rightarrow}\mathcal{D}'$ 
is discontinuous. (This is already stated by L. Schwartz in \cite[p.~158]{SCHW})
\end{proposition}

\begin{proof}
We take the seminorm $T\mapsto |T(\varphi_0)|=|\langle \varphi_0,T\rangle|$ on $\mathcal{D}'$ defined by a test function $\varphi_0$ with $\varphi_0(0)\neq 0$. For given seminorms $p_B$ on $\mathcal{D}'$ and $p_{\tilde{B}}$ on $\mathcal{E'}$, defined by bounded subsets $B$ of $\mathcal{D}$ and $\tilde{B}$ of $\mathcal{E}$, respectively, we search $(S,T)\in\mathcal{E'}\times\mathcal{D}'$ with $|\langle \varphi_0,S\ast T\rangle|>p_{\tilde{B}}(S)p_B(T)$. Since there is a compact subset on $\R^n$ that contains the support of every function in $B$, we can find $x_1$ with $\varphi(x_1)=0$ for every $\varphi\in B$. If we set $T=\delta_{x_1}$ and $S=\delta_{-x_1}$ we get $p_B(T)=\sup_{\varphi\in B}|T(\varphi)|=0$ and, thus 
$$|(S\ast T)(\varphi_0)|=|\varphi_0(0)|>0=p_{\tilde{B}}(S)\sup_{\varphi\in B}|T(\varphi)|=p_B(S)p_{\tilde{B}}(T).$$
\end{proof} 

Since we can take the seminorm $T\mapsto |\langle \varphi_0,T\rangle|$ of the weak topology $\sigma(\mathcal{D}',\mathcal{D})$ on $\mathcal{D}'$, even the mapping $\mathcal{D}'\times\mathcal{E'}\stackrel{\ast}{\rightarrow}(\mathcal{D}',\sigma)$ is discontinuous, if $(\mathcal{D}',\sigma)$ denotes the vector space $\mathcal{D}'$ equipped with the topology $\sigma$.
\par
\begin{remark} In \cite[pp.~133f]{EHR56} L. Ehrenpreis asserts that certain convolutions and multiplications (we number them consecutively \underline{1} - \underline{14}) on the spaces $\mathcal{D}$, $\mathcal{D}'$, $\mathcal{E}$ and $\mathcal{E'}$ are ``continuous bilinear maps''. In fact, only three of the ten listed convolutions are continuous, i.e. \underline{1} $\mathcal{D}\times\mathcal{D}\stackrel{\ast}{\rightarrow}\mathcal{D}$ (continuous by \cite{HSTH01}), \underline{2} $\mathcal{D}\times\mathcal{E'}\stackrel{\ast}{\rightarrow}\mathcal{E'}$ and \underline{3}~$\nolinebreak{\mathcal{E'}\times\mathcal{E'}\stackrel{\ast}{\rightarrow}\mathcal{E'}}$ (both continuous by Proposition 3, since the embedding $\mathcal{D}\subset\mathcal{E'}$ is continuous). The other seven convolutions are discontinuous: The mappings \underline{4} $\mathcal{D}\times\mathcal{E}\stackrel{\ast}{\rightarrow}\mathcal{E}$ and \underline{5}~$\nolinebreak{\mathcal{D}\times\mathcal{D}'\stackrel{\ast}{\rightarrow}\mathcal{E}}$ are discontinuous by Remark 3 and since $\mathcal{E}\hookrightarrow\mathcal{D}'$ is continuous. Furthermore the convolutions \underline{6} $\mathcal{D}\times\mathcal{E'}\stackrel{\ast}{\rightarrow}\mathcal{D}$ and \underline{7}~$\mathcal{E}\times\mathcal{E'}\stackrel{\ast}{\rightarrow}\mathcal{E}$ are discontinuous by Proposition~2 and \underline{8} $\mathcal{D}'\times\mathcal{E'}\stackrel{\ast}{\rightarrow}\mathcal{D}'$ is discontinuous by Proposition~7. 
That the mapping \underline{9} $\mathcal{D}'\times\mathcal{D}\stackrel{\ast}{\rightarrow}\mathcal{D}'$ is also discontinuous can be shown similar to the proof of Proposition 7: Let $0\neq \varphi_0\in\mathcal{D}$, $p$ be a continuous norm on $\mathcal{D}$ and $p_B$ a seminorm  on $\mathcal{D}'$, defined by a bounded subset $B$ of $\mathcal{D}$, such that the support of every $\varphi\in B$ is contained in a compact subset $K$ of $\R^n$. Let $x_0\notin K$, $T=\delta_{x_0}$ and $\varphi\in\mathcal{D}$ with $\int_{R^n} \varphi_0(x)\varphi(x+x_0)\dx x\neq 0$. Then $|\langle \varphi_0,\varphi\ast T \rangle|>0=p_B(T)p(\varphi)$. 
\par If we have $0\neq \varphi_0\in\mathcal{D}$, a seminorm $p$ on $\mathcal{E'}$ and a seminorm $p_{m,K}$ on $\mathcal{E}$ we choose $x_0\in\R^n$ and $f\in\mathcal{E}$ such that $f|_{K}=0$ and $\int_{R^n} \varphi_0(x)f(x+x_0)\dx x\neq 0$. We have $|\langle \varphi_0,f\ast T \rangle|>0=p_{m,K}(f)p(T)$ and, thus, the convolution \underline{10} $\mathcal{E}\times\mathcal{E'}\stackrel{\ast}{\rightarrow}\mathcal{D}'$ is also not continuous. \par
Out of the four multiplications listed by L. Ehrenpreis, only two are continuous. The mappings \underline{11} $\mathcal{D}\times\mathcal{D}\stackrel{\cdot}{\rightarrow}\mathcal{D}$ and \underline{12} $\mathcal{E}\times\mathcal{E}\stackrel{\cdot}{\rightarrow}\mathcal{E}$ are continuous by \cite{HSTH01} and Proposition~3, respectively. That the multiplication \underline{13} $\mathcal{D}'\times\mathcal{E}\stackrel{\cdot}{\rightarrow}\mathcal{D}'$ is discontinuous is contained in Proposition~6. That the multiplication \underline{14} $\mathcal{D}\times\mathcal{D}'\stackrel{\cdot}{\rightarrow}\mathcal{E'}$ is not continuous can be seen similar to the previous proofs: Take the seminorm $T\mapsto |\langle 1,T\rangle|$ on $\mathcal{E'}$ and let $p$ be a norm on $\mathcal{D}$ and $p_B$ a seminorm  on $\mathcal{D}'$, defined by a bounded subset $B$ of $\mathcal{D}$, such that the support of every $\varphi\in B$ is contained in a compact subset $K$ of $\R^n$. For $x_0\notin K$ we set $T=\delta_{x_0}$ and take $\varphi\in\mathcal{D}$ with $\varphi(x_0)\neq 0$. Hence $|\langle 1,\varphi T\rangle|=|\varphi(x_0)|>0=p_B(T)p(\varphi)$. In the end only five of the fourteen mappings are (jointly) continuous.
\end{remark}

\begin{remark}
Let us make an attempt to explain why all the multiplications, that are defined on the product of a function and a distribution space and take values in a distribution space, and all the regularizations, i.e. convolutions defined on the product of a function and a distribution space and taking values in a function space, are discontinuous. \par A seminorm on a function space merely measures the derivatives of a function up to a certain order, but for a $\mathcal{C}^\infty$--function $f$ and a distribution $T$, the distribution $f\cdot T$ and the function $f\ast T$ can inherit derivatives of $f$ of arbitrary order, since we can take $T$ an arbitrary derivative of the Dirac delta---which actually led us to the counterexamples in the concerning proofs. The fact that we cannot estimate derivatives of higher order with derivatives of lesser order hence implies discontinuity.

\end{remark}
Finally, we collect the topological algebras occurring in this paper. The algebras with continuous multiplication were $\mathcal{D}$, $\mathcal{S}$, $\bp$, $\mathcal{D}_{L^p}$ ($1\leq p\leq \infty$), $\om$ and $\mathcal{E}$. The algebras with continuous convolution were $\mathcal{D}$, $\mathcal{S}$, $\mathcal{D}_{L^1}$, $\mathcal{E'}$, $\oc'$ and $\mathcal{D}'_{L^1}$.
\\\par
{\textbf{Acknowledgment.} Let me thank Norbert Ortner and Peter Wagner for essential contributions to this paper and in particular Norbert Ortner for his continuous support. I also thank Christian Bargetz for his collaboration and fruitful discussions concerning the topic.}

\bibliographystyle{alpha}
\bibliography{../literatur}

\end{document}